\documentclass{amsart}
\usepackage{amssymb, amsmath}
\usepackage{verbatim}
\newtheorem{theorem}{Theorem}[section]
\newtheorem{lemma}[theorem]{Lemma}
\newtheorem{proposition}[theorem]{Proposition}

\newtheorem{corollary}[theorem]{Corollary}

\theoremstyle{definition}

\theoremstyle{remark}
\newtheorem{remark}[theorem]{Remark}

\numberwithin{equation}{section}

\newcommand{\LP}[2]{L^{#1}(X_{#2}, \Sigma_{#2}, m_{#2})}

\newcommand{\M}{\mathcal{M}}

\newcommand{\tM}{\widetilde{\mathcal{M}}}

\newcommand{\I}{1\!{\mathrm l}}

\begin{document}

\title{Multipliers on Noncommutative Orlicz Spaces}

\author{Louis E. Labuschagne}
\address{School for Comp., Stat. \& Math. Sci., Fac. of Nat. Sci.,
North-West Univ. (Potch. Campus) - Box 209, Pvt Bag X6001, 2520 Potchefstroom,
South Africa} 
\email{louis.labuschagne@nwu.ac.za}
\date{\today}

\thanks{This work is based on research supported by the National Research Foundation. Any opinion, findings and conclusions or recommendations expressed in this material, are those of the author, and therefore the NRF do not accept any liability in regard thereto.}

\begin{abstract} 
We establish very general criteria for the existence of multiplication operators between noncommutative Orlicz spaces $L^{\psi_0}(\tM)$ and $L^{\psi_1}(\tM)$. We then show that these criteria contain existing results, before going on to briefly look at the extent to which the theory of multipliers on Orlicz spaces differs from that of $L^p$-spaces. In closing we describe the compactness properties of such operators.
\end{abstract}

\maketitle

\section{Preliminaries}

General von Neumann algebraic notation will be based on that of \cite{BRo},  
\cite{Tak} with $\M$ denoting a von Neumann algebra and $\I$ the identity  
element thereof. As regards $L_p$-spaces we will use \cite{Tp} and \cite{FK} 
as basic references for the non-commutative context. In this paper we will 
restrict attention to the case of semifinite von Neumann algebras. The fns 
trace of such an algebra $\M$ will be denoted by $\tau_{\M} = \tau$. The projection 
lattice of a von Neumann algebra $\M$ will be denoted by $\mathbb{P}(\M)$.

By the term an \emph{Orlicz function} we understand a convex function 
$\varphi : [0, \infty) \to [0, \infty]$ satisfying $\varphi(0) = 0$ and $\lim_{u \to \infty} 
\varphi(u) = \infty$, which is neither identically zero nor infinite valued on all of 
$(0, \infty)$, and which is left continuous at $b_\varphi = \sup\{u > 0 : \varphi(u) < 
\infty\}$. These axioms ensure that any Orlicz function 
must also be increasing, and continuous on $[0, b_\varphi]$. 

Within the category of Orlicz functions there is an especially ``nice'' class of Orlicz functions called the $N$-functions. In essence an Orlicz function $\varphi$ is an $N$-function if it is continuous and satisfies $\lim_{t\to 0}\frac{\varphi(t)}{t} = 0$ and $\lim_{t\to \infty}\frac{\varphi(t)}{t} = \infty$

It is worth pointing out that each Orlicz function is ``almost'' an $N$-function in the following sense:

\begin{lemma}\label{Nfn}
Let $\varphi$ be an Orlicz function. For any $0 < q < 1$ we then have that $$\lim_{t\to \infty}\frac{\varphi(t)}{t^q} = \infty \quad \mbox{and} \quad \lim_{t\to 0}\frac{\varphi(t)}{t^{1/q}} = 0.$$
\end{lemma}

\begin{proof}
Let $t_0 > 0 $ be given. By convexity and the fact that $\varphi(0)=0$, we will have that $\varphi(st_0) \leq s\varphi(t_0)$ for all $0 \leq s 
\leq 1$. Equivalently $\varphi(st_0) \geq s\varphi(t_0)$ for all $s \geq 1$. Notice that the second limit formula is obviously true if $a_\varphi > 
0$. We may therefore clearly assume that $a_\varphi = 0$. But in that case we can pick $t_0 > 0$ so that $\infty > \varphi(t_0) > 0$. If we combine this with the former inequalities, we therefore get that $$\liminf_{s\to\infty} \frac{\varphi(s)}{s} \geq \frac{\varphi(t_0)}{t_0} > 0 \quad\mbox{and}\quad \infty > \frac{\varphi(t_0)}{t_0} \geq \limsup_{r\to 0} \frac{\varphi(r)}{r}.$$These inequalities in turn enable us to show that for any $0 < q < 1$ we will have that $\liminf_{s\to\infty} \frac{\varphi(s)}{s^q} = \infty$ and $\limsup_{r\to 0} \frac{\varphi(r)}{r^{1/q}} = 0$. This is sufficient to prove the claims. 
\end{proof}

Each Orlicz function $\varphi$ induces a complementary Orlicz function $\varphi^*$ which is defined by 
$\varphi^*(u) = \sup_{v > 0}(uv - \varphi(v))$. The pair $\varphi$ and $\varphi^*$ satisfy the following Hausdorff-Young inequality:
$$st \leq \varphi(s) + \varphi^*(t) \qquad s,t \geq 0.$$
The formal ``inverse'' $\varphi^{-1}: [0, \infty) 
\to [0, \infty]$ of an Orlicz function is defined by the formula
$$\varphi^{-1}(t) = \sup\{s : \varphi(s) \leq t\}.$$It is however only really in the case where $0 = \inf\{u > 0 : \varphi(u) > 0\}$ and $b_\varphi = \infty$, where this is an inverse in the strict sense of the word.

If we write $\LP{0}{}$ for the space of measurable 
functions on some $\sigma$-finite measure space $(X, \Sigma, m)$, then given an Orlicz function $\varphi$, the Orlicz space 
$\LP{\varphi}{}$ associated with  $\varphi$ is defined to be the set $$L^{\varphi} = \{f \in 
L^0 : \varphi(\lambda |f|) \in L^1 \quad \mbox{for some} \quad \lambda = \lambda(f) > 0\}.$$
This space turns out to be a linear subspace of $L^0$ which becomes a Banach space when 
equipped with the so-called Luxemburg-Nakano norm 
$$\|f\|_\varphi = \inf\{\lambda > 0 : \|\varphi(|f|/\lambda)\|_1 \leq 1\}.$$An equally natural, but ultimately equivalent, norm for 
this space is the so-called Orlicz norm, given by the formula
$$\|f\|^0_\varphi = \sup\{\int |fg| \, \mathrm{d}m : f \in L^{\varphi^*}, \|f\|_\varphi \leq 1\}. $$

The Orlicz spaces corresponding to the specific Orlicz functions $\varphi_p(t) = t^p$ ($\infty > p \geq 1$) are of course 
nothing but the $L^p$-spaces. In the category of Orlicz spaces various classes of spaces may be distinguished by imposing various growth conditions on the underlying Orlicz function. We briefly mention those we will have occasion to use. We say that $\varphi$ satisfies $\Delta_2$ \emph{for all} $u$ if there exists a positive constant $K$ such that $\varphi(2u) \leq K \varphi(u)$ for all $u > 0$. Within this class we find those Orlicz functions $\varphi$ which satisfy the $\Delta'$ condition which states that there exists a constant $C > 0$ and some $u_0$ such that $$\varphi(st) \leq C\psi(s)\psi(t) \qquad \mbox{for all} \quad s, t \geq u_0.$$
Similarly $\psi$ is said to satisfy condition $\nabla'$ if there exists a constant $b > 0$ and some $u_0$ such that $$\psi(bst) \geq \psi(s)\psi(t) \qquad \mbox{for all} \quad s, t \geq u_0.$$If one of the above conditions holds for the case $u_0 = 0$, that condition is said to hold \emph{globally}.

The $\Delta'$ condition as formulated above, can variously be shown to be equivalent to the existence of a constant 
$a > 0$ for which $\psi(ast) \leq \psi(s)\psi(t)$ for all $s, t \geq u_0$. To see this notice that by convexity and the fact that $\psi(0) = 0$, we have that $\psi(st) \leq s\psi(t)$ for all $0 \leq s \leq 1$ and all $t \geq 0$. So if one assumes that $C>1$, one may set $a = \frac{1}{C}$, and use this fact to conclude from the $\Delta'$-condition that $\psi(ast) \leq \frac{1}{C}\psi(st) \leq \psi(s)\psi(t)$. Conversely if for some $a > 0$ we have that $\psi(ast) \leq \psi(s)\psi(t)$ for all $s, t \geq u_0$, then on assuming that $a$ is small enough to ensure that $\frac{1}{a^2} \geq u_0$, it follows that $\psi(st) \leq \psi(\frac{1}{a}s)\psi(t) \leq \psi(\frac{1}{a^2})\psi(s)\psi(t)$ for all $s, t \geq u_0$.

Now let $\M$ be a semifinite von Neumann algebra equipped with some faithful normal semifinite ($fns$) trace $\tau$. Given an 
Orlicz function $\varphi$, we may define noncommutative Orlicz spaces by exploiting the very powerful theory of Noncommutative Banach Function 
Spaces. \cite{DDdP}. We briefly describe the basic idea behind this theory. Essentially the space of all $\tau$-measurable operators 
$\widetilde{\M}$ (equipped with the topology of convergence in measure) plays the role of 
$L^0$. Given an element $f \in \widetilde{\M}$ and $t \in [0, \infty)$, the generalised singular 
value $\mu_t(f)$ is defined by $\mu_t(f) = \inf\{s \geq 0 : \tau(\I - e_s(|f|)) \leq t\}$ 
where $e_s(|f|)$, $s \in \mathbb{R}$, is the spectral resolution of $|f|$. The function $t \to 
\mu_t(f)$ will generally be denoted by $\mu(f)$. For details on the generalised singular value 
see \cite{FK}. (This directly extends classical notions where for any $f \in \LP{\infty}{}$, 
the function $(0, \infty) \to [0, \infty] : t \to \mu_t(f)$ is known as the decreasing 
rearrangement of $f$.) In the context of the theory of Dodds, Dodds and de Pagter \cite{DDdP} 
the noncommutative space $L^\varphi(\widetilde{\M})$ is formally defined to be the set
$$L^\varphi(\widetilde{\M}) = \{f \in \widetilde{\M} : \mu(f) \in 
L^{\varphi}(0, \infty)\}.$$Since for any Orlicz function $\varphi$, the Orlicz 
space $L^\varphi(0, \infty)$ is known to be a rearrangement invariant Banach Function space 
with the norm having the Fatou Property \cite[Theorem 4.8.9]{BS}, the theory of Dodds, Dodds and de Pagter 
then informs us that the space $L^\varphi(\widetilde{\M})$ defined above is a Banach space under the norm 
$\|f\|_\varphi = \|\mu(f)\|_\varphi$, which in addition injects continuously into the space of 
$\tau$-measurable operators $\tM$.

We close this introductory section by briefly mentioning a few technical results from \cite{LM}, the point of which are that if one 
follows a na\"ive approach to defining noncommutative Orlicz spaces by simply replacing $L^0$ with $\tM$ and the integral $\int \mathrm{d}m$ 
with $\tau$, one would essentially get the same space as the one produced by the theory of Dodds, Dodds and de Pagter. These facts will prove to 
be a useful tool later on.

\begin{lemma}[\cite{LM}]\label{DPvsKlemma} Let $\varphi$ be an Orlicz function and $f \in \widetilde{\M}$ a $\tau$-measurable element for which $\varphi(|f|)$ is again $\tau$-measurable. Extend $\varphi$ to a function on $[0, \infty]$ by setting $\varphi(\infty) = \infty$. Then $\varphi(\mu_t(f)) = \mu_t(\varphi(|f|))$ for any $t \geq 0$. Moreover $\tau(\varphi(|f|)) = \int_0^\infty \varphi(\mu_t(|f|))\, \mathrm{d}t$.
\end{lemma}

\begin{proposition}[\cite{LM}]\label{DPvsK} Let $\varphi$ be an Orlicz function and let $f \in \widetilde{\M}$ be given. There exists some 
$\alpha > 0$ so that $\int_0^\infty \varphi(\alpha\mu_t(|f|))\, \mathrm{d}t < \infty$ if and only if there exists 
$\beta > 0$ so that $\varphi(\beta|f|) \in \widetilde{\M}$ and $\tau(\varphi(\beta|f|)) < \infty$. Moreover $$\|\mu(f)\|_\varphi = \inf\{\lambda > 0 : \varphi\left(\frac{1}{\lambda}|f|\right) \in \widetilde{\M}, \tau\left(\varphi\left(\frac{1}{\lambda}|f|\right)\right) \leq 1\}.$$
\end{proposition}

We close this section by formulating one more fact regarding Orlicz spaces that will also prove to be useful later 
on. (This is a special case of known results in \cite{DDdP3}). The short proof may be found in \cite{LM}. 

\begin{proposition}\label{kothe}
Let $\varphi$ be an Orlicz function and $\varphi^*$ its complementary function. Then $L^{\varphi^*}(\widetilde{\M})$ equipped with the norm $\|\cdot\|^0_{\varphi^*}$ defined by $$\|f\|^0_{\varphi^*} = \sup \{\tau(|fg|) : g \in L^{\varphi}(\widetilde{\M}), \|g\|_\varphi \leq 1\} \qquad f \in L^{\varphi^*}(\widetilde{\M})$$is the K\"{o}the dual of $L^{\varphi}(\widetilde{\M})$. That is $$L^{\varphi^*}(\widetilde{\M}) = \{f \in \widetilde{\M} : fg \in L^1(\M, \tau) \text{ for all } g \in L^{\varphi}(\widetilde{\M})\}.$$Consequently
$$|\tau(fg)| \leq \|f\|^0_{\varphi^*}\cdot\|g\|_\varphi \quad\mbox{for all}\quad f \in L^{\varphi^*}(\widetilde{\M}), g \in L^{\varphi}(\widetilde{\M}).$$
\end{proposition}

In the paper \cite{LM} Labuschagne and Majewski recently provided general criteria for the existence of a composition operator between two possibly different Orlicz spaces $L^{\psi_0}$ and $L^{\psi_1}$. The appearance of this result raises the question of whether a similar situation pertains as far as multiplication operators are concerned. In other words what general criteria will guarantee the existence of multiplication operators between possibly \emph{different} Orlicz spaces? Furthermore in this general context, how closely may we reasonably expect the theory of multiplication operators to parallel the theory of composition operators? Is the resultant theory nothing more than a minor technical modification of the $L^p$ theory, or may we expect some truly exotic behaviour? In the ensuing sections we will attempt to cast some light on these and related questions, before closing with a consideration of the compactness properties of multiplication operators on Orlicz spaces.

\section{Existence of Multipliers}

We first provide very general criteria for the existence of multiplication operators between possibly different Orlicz spaces. We will see that all that is neede to guarantee the existence of such an operator, is a \emph{generalised Hausdorff-Young inequality}. These criteria turn out to be more general that any of the criteria that seems to be known currently. 
 
We briefly review some known results before proving our eixtence theorem. The result stated below is basically just a compilation of Theorems 13.7 and 13.8 of \cite{KR}, which have been slightly reformulated to better fit the present context.

\begin{theorem}[\cite{KR}]\label{krasrut}
Let $\zeta$, $\varphi_1$, and $\varphi_2$ be N-functions and let $G$ be a closed bounded subset of $\mathbb{R}^n$ equipped with Lebesgue measure. For any $f \in L^\zeta(G)$ and any $g \in L^{\varphi_1}(G)$ we will have that $fg \in L^{\varphi_2}(G)$ whenever any of the following conditions hold:
\begin{itemize}
\item There exist constants $\alpha$ and $\beta$ for which $$\varphi_2\circ\zeta(u) < \varphi_1(\alpha u), \qquad \varphi_2\circ\zeta^*(u) < \zeta(\beta u) \qquad \mbox{for all} \quad u \geq u_0.$$
\item The function $\varphi_2$ satisfies the $\Delta'$ condition globally and there exist constants $\alpha$ and $\beta$ for which $$\zeta(\alpha\varphi_2(u)) < \varphi_1(u), \qquad \zeta^*(\beta\varphi_2(u)) < \zeta(u) \qquad \mbox{for all} \quad u \geq u_0.$$
\end{itemize}
\end{theorem}

We are now ready to provide what we believe to be the most general existence criteria for a multiplication operator.

\begin{theorem}\label{existence}
Let $\zeta, \varphi_1, \varphi_2$ be Orlicz functions for which there exist positive constants $M, \alpha, \beta, \gamma$ such that
$$uvw \leq M\left[\varphi_2^*(\alpha u) + \varphi_1(\beta v) + \zeta(\gamma w)\right]\quad \mbox{for all} \quad u, v, w \geq 0.$$ 
For any $f\in L^\zeta(\tM)$ and any $g\in L^{\varphi_1}(\tM)$, we then have that $fg \in L^{\varphi_2}(\tM)$.
\end{theorem}

\begin{proof}Let $f\in L^\zeta(\tM)$, $g\in L^{\varphi_1}(\widetilde{M})$ and $h\in L^{\varphi_2^*}(\tM)$ be given. For the sake of argument suppose that $a_{\varphi_i} = 0$, $b_{\varphi_i}=\infty$ (for $i = 1,2$), and that $\tau(\zeta(f)) < \infty$, $\tau(\varphi_1(g)) < \infty$ and $\tau(\varphi_2^*(h)) < \infty$. By the given inequality, we have that $$\mu_t(fgh) \leq \mu_{t/3}(f)\mu_{t/3}(g)\mu_{t/3}(h) \leq M\left[\zeta(\mu_{\alpha t/3}(f)) + \varphi_1(\mu_{\beta t/3}(g)) + \varphi_2^*(\mu_{\gamma t/3}(h))\right].$$ So by Lemma \ref{DPvsKlemma}, 
\begin{eqnarray*}
\tau(|fgh|) &=& \int_0^\infty \mu_t(fgh))dt\\
&\leq& M \int_0^\infty \left[\zeta(\mu_{\alpha t/3}(f)) + \varphi_1(\mu_{\beta t/3}(g)) + \varphi_2^*(\mu_{\gamma t/3}(h))\right] dt\\
&=& M \left[\frac{3}{\alpha}\int_0^\infty\zeta(\mu_{t}(f))dt + \frac{3}{\beta}\int_0^\infty \varphi_1(\mu_{t}(g))dt +\frac{3}{\gamma}\int_0^\infty \varphi_2^*(\mu_{t}(h))dt\right]\\
&=& M\left[\frac{3}{\alpha}\tau(\zeta(f)) + \frac{3}{\beta}\tau(\varphi_1(g)) + \frac{3}{\gamma}\tau(\varphi_2^*(h))\right]\\
&<& \infty. 
\end{eqnarray*}
Thus $fgh \in L^1(\M)$. So by Proposition \ref{kothe}, we have $fg \in L^{\varphi_2}(\M)$ as required.
\end{proof}

\begin{corollary}\label{compmult}
Let $\psi, \varphi_1, \varphi_2$ be Orlicz functions.
If either of the following two conditions hold, we have that $fg \in L^{\varphi_2}(\tM)$ for any $f\in L^\zeta(\tM)$ 
and any $g\in L^{\varphi_1}(\tM)$. 
\begin{enumerate}
\item[(a)] $\zeta = \varphi_2\circ\psi^*$ is again an Orlicz function, and $\varphi_2\circ\psi = \varphi_1$.
\item[(b)] $\varphi_2$ satisfies $\Delta'$ globally, $\zeta = \psi^*\circ\varphi_2$ is again an Orlicz function, and $\psi\circ\varphi_2 = \varphi_1$.
\end{enumerate}
\end{corollary}

\begin{proof}If condition (a) holds, then
\begin{eqnarray*}
uvw &\leq& \varphi_2^*(u) + \varphi_2(vw)\\
&\leq& \varphi_2^*(u) + \varphi_2(\psi(v) + \psi^*(w))\\
&\leq& \varphi_2^*(u) + \varphi_2(\frac{1}{2}\psi(2v) + \frac{1}{2}\psi^*(2w))\\
&\leq& \varphi_2^*(u) + \frac{1}{2}\left[\varphi_2(\psi(2v)) + \varphi_2(\psi^*(2w))\right]\\
&\leq& \varphi_2^*(u) + \frac{1}{2}\left[\varphi_1(2v)) + \zeta(2w)\right].
\end{eqnarray*} 
If on the other hand condition (b) holds, then using the fact that $\varphi_2(avw) \leq \varphi_2(v)\varphi_2(w)$ for some $a > 0$, a similar argument yields
\begin{eqnarray*}
uvw &\leq& \varphi_2^*(u) + \varphi_2(vw)\\
&\leq& \varphi_2^*(u) + \varphi_2(\frac{1}{a}v)\varphi_2(w)\\
&\leq& \varphi_2^*(u) + [\psi(\varphi_2(\frac{1}{a}v)) + \psi^*(\varphi_2(w))]\\
&\leq& \varphi_2^*(u) + \varphi_1(\frac{1}{a} v)) + \zeta(w).
\end{eqnarray*} 
 
\end{proof}

\begin{remark}
It remains to show that the above existence criteria truly are more general than those provided in \cite{KR}. In this regard notice that the results in \cite{KR} summarised in Theorem \ref{krasrut}, are formulated for finite measure spaces. In the case of finite measure spaces one can get by with more relaxed criteria in that neither the inequalities stated in the first condition, the $\Delta'$ condition in the second need to hold globally, but rather only for $u$ larger than some $u_0 \geq 0$. For infinite measures we need $u_0=0$ for those results to work. In this context (with $u_0$ replaced by 0), both the above conditions are indeed contained in Theorem \ref{existence} and its corollary. To see this, notice that if the first condition in Theorem \ref{krasrut}, we may use convexity and the usual Hausdorff-Young inequality to conclude that 
\begin{eqnarray*}
uvw &\leq& \varphi_2^*(u) + \varphi_2(vw)\\
&\leq& \varphi_2^*(u) + \varphi_2(\zeta(v) + \zeta^*(w))\\
&\leq& \varphi_2^*(u) + \varphi_2(\frac{1}{2}\zeta(2v) + \frac{1}{2}\zeta^*(2w))\\
&\leq& \varphi_2^*(u) + \frac{1}{2}\left[\varphi_2(\zeta(2v)) + \varphi_2(\zeta^*(2w))\right]\\
&\leq& \varphi_2^*(u) + \frac{1}{2}\left[\varphi_1(2\alpha v)) + \zeta(2\beta w)\right].
\end{eqnarray*}

If on the other hand $\varphi_2$ satisfies $\Delta'$ globally, we may find $a > 0$ so that $\varphi_2(vw) \leq \varphi_2(\frac{1}{a}v)\varphi_2(w)$ for all $v, w \geq 0$. If therefore the second condition Theorem \ref{krasrut} holds, we may then use convexity and the usual Hausdorff-Young to conclude that 
\begin{eqnarray*}
uvw &\leq& \varphi_2^*(u) + \varphi_2(vw)\\
&\leq& \varphi_2^*(u) + \varphi_2(\frac{1}{a}v)\varphi_2(w)\\
&\leq& \varphi_2^*(u) + \frac{1}{\alpha\beta}[\alpha\varphi_2(\frac{1}{a}v).\beta\varphi_2(w)]\\
&\leq& \varphi_2^*(u) + \frac{1}{\alpha\beta}[\zeta(\alpha\varphi_2(\frac{1}{a}v)) + \zeta^*(\beta\varphi_2(w))]\\
&\leq& \varphi_2^*(u) + \frac{1}{\alpha\beta}\left[\varphi_1(\frac{1}{a} v)) + \zeta(w)\right].
\end{eqnarray*}

\end{remark}

\section{Common multiplier-based techniques}

\subsection{Rescaling multipliers}

For $L^p$-spaces any multiplier $g$ from say $L^p$ to $L^q$ can be ``rescaled'' to produce a multiplier from $L^{pr}$ to $L^{qr}$. If $g = u|g|$ is the polar decomposition of $g$, one simply replaces $g$ with $u|g|^{1/r}$ to achieve this. But what about Orlicz spaces? Suppose that $\psi$, $\varphi_1$ and $\varphi_2$ are Orlicz functions related by the equation $\psi \circ \varphi_1 = \varphi_2$ and suppose that for some measure space $(X, \Sigma, \nu)$, $f \geq 0$ induces a multiplication operator from $L^{\psi}(\nu)$ to $L^1(\nu)$. When can $f$ be similarly rescaled to produce an element which induces a related multiplication map from $L^{\varphi_1}$ to $L^{\varphi_2}$ with similar properties to the original? In the case where $\varphi_2$ satisfies $\Delta'$ globally, we have the following result.

\begin{lemma}[\cite{LM}]
Let $\psi, \varphi_1, \varphi_2$ be Orlicz functions for which $\psi\circ\varphi = \zeta$, and let $\M$ be a semifinite von Neumann algebra with fns trace $\tau$.  For any $a \in L^{\zeta}(\widetilde{\M})$ with $\|a\|_{\zeta} < 1$, we have that $\varphi(|a|) \in L^{\psi}(\widetilde{\M})$ and $\|\varphi(|a|)\|_{\psi} \leq \|a\|_{\varphi_1}$.
\end{lemma}

\begin{proposition}\label{scalemult}
Let $\psi$, $\varphi_1$ and $\varphi_2$ be Orlicz functions related by $\psi \circ \varphi_1 = \varphi_2$, and for which $\zeta = \psi^*\circ\varphi_2$ is again an Orlicz function. If 
$\varphi_2$ satisfies $\Delta_2$, then for any $g \in L_+^\zeta(\tM)$, the 
element $\varphi_2(g)$ will belong to $L^{\psi^*}(\tM)$, and hence induce a multiplication operator from $L^{\psi}(\tM)$ to $L^1(\tM)$. If on the 
other hand $\varphi_2$ satisfies $\Delta'$ globally, then for any  $f \in L^{\psi^*}_+(\tM)$, the function $\varphi_2^{-1}(f)$ will induce a multiplication operator from $L^{\varphi_1}(\tM)$ to $L^{\varphi_2}(\tM)$.
\end{proposition}

\begin{proof}
Let $g \in L_+^\zeta(\tM)$ be given, and select $\alpha > 0$ so that $\|\alpha g\|_\zeta < 1$. Then $\varphi(\alpha g) \in L^{\psi^*}_+(\tM)$ by the lemma. Since $\varphi_2$ satisfies the $\Delta_2$ condition, we can find a constant $K$ so that $\varphi_2(2u) \leq K \varphi_2(u)$ for all $u\geq 0$. Clearly $\alpha > \frac{1}{2^N}$ for some $N$. Then $\varphi_2(\frac{1}{\alpha}u) \leq \varphi_2(2^Nu) \leq K^N \varphi_2(u)$ for all $u$, or equivalently $\varphi_2(u) \leq K^N \varphi_2(\alpha u)$. Therefore $\varphi_2(g) \leq K^N \varphi_2(\alpha g)$. This in turn ensures that $\varphi(g) \in L^{\psi^*}_+(\tM)$.    
As far as the second claim is concerned, notice that $\Delta'$ implies $\Delta_2$ and hence that $\varphi_2^{-1}$ is a ``proper'' inverse function. For any $0 < s < 1$, convexity ensures that $\zeta(s\varphi_2^{-1}(f)) = \psi^*\circ\varphi_2(s\varphi_2^{-1}(f)) \leq \psi^*(s\varphi_2(\varphi_2^{-1}(f)) = \psi^*(sf)$. This inequality in turn enables us to conclude that $\varphi_2^{-1}(f) \in L_+^\zeta(\tM)$. The claim now follows from Corollary \ref{compmult}.
\end{proof}

\subsection{Spaces constructed from equivalent measures} For $L^p$-spaces it is precisely the ability to rescale multipliers (mentioned at the start of this section), that ensures that $L^p$-spaces produced by equivalent measures, are linearly isometric. To see this let $(X, \Sigma)$ be a measure space equipped with two $\sigma$-finite measures $\nu_1, \nu_2$ which are equivalent in the sense that $\nu_1 << \nu_2$ and $\nu_2 << \nu_1$. Then for any $p > 0$, $f \to f\left(\frac{d\nu_1}{d\nu_2}\right)^{1/p}$ defines a linear isometry from $L^p(\nu_1)$ onto $L^p(\nu_2)$. But what about Orlicz spaces? $f \to f\left(\frac{d\nu_1}{d\nu_2}\right)$ defines a linear isometry from $L^1(\nu_1)$ onto $L^1(\nu_2)$. Given a general Orlicz function $\varphi$, when can we rescale this multiplication operator to produce a multiplication map which isomorphically identifies $L^\varphi(\nu_1)$ with $L^\varphi(\nu_2)$? In our investigation of the plausibility of such a state of affairs, we will restrict our attention to classical Orlicz spaces. However since our primary interest is still in understanding which techniques may reasonably be expected to carry over to the noncommutative setting, we we will impose no restriction on the underlying measure spaces other than those that are absolutely necessary. (The category of commutative von Neumann algebras is known to correspond up to $*$-isomorphism to $L^\infty$-spaces on localizable measure spaces. \cite{Sak}.) Thus the simplification that may be achieved by restricting attention to standard Borel spaces is not at our disposal. 

The following proposition goes some way to answering the above question.

\begin{proposition}\label{equivspace}
Let $(X, \Sigma)$ be a measure space equipped with two $\sigma$-finite measures 
$\nu_1, \nu_2$ which are equivalent in the sense that $\nu_1 << \nu_2$ and $\nu_2 << \nu_1$.
If $\varphi$ satisfies $\Delta'$, then $f \to \varphi^{-1}\left(\frac{d\nu_1}{d\nu_2}\right)f$ defines a continuous map from $L^\varphi(\nu_1)$ to $L^\varphi(\nu_2)$. If $\varphi$ satisfies $\nabla'$, then this same map is bounded below on its domain. 
\end{proposition}

\begin{proof}
Suppose first that $\varphi \in \Delta'$. That means we can find some $a>0$ for which $$\varphi(auv) \leq \varphi(u)\varphi(v)$$ for all $u,v \geq 0$. As before the fact that in this case $\varphi_2^{-1}$ is a ``proper'' inverse function ensures that $\varphi^{-1}\left(\frac{d\nu_1}{d\nu_2}\right)$ is well-defined. From the first centred equation we may now conclude that for any $f \in L^\varphi$ and any $s>0$, we have  $$\varphi\left(\frac{a}{s}\varphi^{-1}\left(\frac{d\nu_1}{d\nu_2}\right)f\right) \leq \frac{d\nu_1}{d\nu_2}\varphi(\frac{1}{s}f).$$This in turn leads to the inequality $$\int \varphi\left(\frac{a}{s}\varphi^{-1}\left(\frac{d\nu_1}{d\nu_2}\right)f\right) d\nu_2 \leq \int\varphi(\frac{1}{s}f)d\nu_1,$$which in turn ensures that we have $$\|\varphi^{-1}\left(\frac{d\nu_1}{d\nu_2}\right)f\|_{\varphi, 2} \leq a\|f\|_{\varphi, 1}$$for the Luxemburg norm.

If on the other hand $\varphi$ satisfies $\nabla'$, a similar argument shows that we can then find some $b>0$ for which the Luxemburg norm satisfies $$\|\varphi^{-1}\left(\frac{d\nu_1}{d\nu_2}\right)f\|_{\varphi, 2} \geq b\|f\|_{\varphi, 1}.$$
\end{proof}

Proposition \ref{scalemult} suggests that $\varphi^{-1}\left(\frac{d\nu_1}{d\nu_2}\right)$ is indeed the appropriate function to use when trying to build an isomorphism between the Orlicz spaces by means of the Radon-Nikodym derivative. However for very general measure spaces, it is difficult to see how one may obatin a result of the above type without assuming $\varphi \in \Delta' \cap \nabla'$. Unfortunately there is only one problem with that restriction. As we can see from the result below, it yields no new information! 

\begin{theorem}
An Orlicz function $\psi$ belongs to $\Delta' \cap \nabla'$ if and only if $\psi \sim t^p$ for some $p \geq 1$. 
\end{theorem}

\begin{proof}
This result is proved on page 31 of \cite{RR} under the assumption that $\varphi$ is an $N$-function. Under that assumption the conclusion is that $\varphi \sim t^p$ where $p > 1$. However the only point where the proof presented there uses the assumption that $\varphi$ is an $N$-function, is precisely to show that $p > 1$. So if we dispense with this restriction and follow the proof of \cite{RR} for general Orlicz functions, we are able to show that there exist constants $0 < a_1 \leq a_2$, some $x_0 \geq 0$ and $p > 0$ so that $$(a_1x)^p \leq \varphi(x) \leq (a_2x)^p \quad \mbox{for all} \quad x \geq x_0.$$It remains to show that $p \geq 1$. This fact now follows fairly directly from Lemma \ref{Nfn}.   
\end{proof}

The unsettling conclusion we are left with, is that in the most general setting, it is not at all clear that ``equivalent'' measures will indeed produce equivalent Orlicz spaces.

\subsection{Multipliers and Composition Operators}

In the case of $L^p$ spaces the theory of multipliers closely parallels the theory of composition operators. The reason for this is that up to isometric inclusions, there is a sense in which any composition operator between $L^p$ spaces is in fact induced by a multiplier.
 
To see this let $(X_i, \Sigma_i, m_i)$ $(i = 1, 2)$ be measure spaces and let
$T : X_2 \rightarrow X_1$ be a given non-singular
measurable transformation from $X_2$
into $X_1$. Suppose for the sake of simplicity suppose that $\infty>p\geq q\geq1$. 
If the process $f \rightarrow f \circ T$ yields a bounded
linear operator from $L^p(X_1, m_1)$ to $L^q(Y, m_2) \subset
L^q(X_2, m_2)$, we call the resultant operator a \emph{composition 
operator from $L^p(X_1, m_1)$ to $L^q(X_2, m_2)$}.

In the following let $\Sigma_T$ be the $\sigma$-subalgebra of
$\Sigma_2$ generated by sets of the form $T^{-1}(E)$ where $E \in
\Sigma_1$, and let $f_T = \frac{dm_2\circ T^{-1}}{dm_1}$. Our composition operator is then made up of the following
processes:
\begin{enumerate}
\item[(I)] Restricting to the support $Z$ of $m_2 \circ T^{-1}$:
$L^p(X_1, m_1) \rightarrow L^p(Z, m_1|_Z): f \mapsto f|_Z$

\item[(II)] Scaled multiplication by the Radon-Nikodym derivative:
$L^p(Z, m_1|_Z) \rightarrow L^q(Z, m_1|_Z): f \mapsto ff_T^{1/q}$
scaled 
\item[(III)] Changing measures: $L^q(Z, m_1|_Z) \rightarrow L^q(Z, m_2 \circ
T^{-1}): ff_T^{1/q} \mapsto f$

\item[(IV)] Isometric equivalence of spaces: $L^q(Z, \Sigma_1^Z, m_2 \circ T^{-1})
\rightarrow L^q(X_2, \Sigma_T, m_2): f \mapsto f \circ T$ (Here
$\Sigma_1^Z = \{E \in \Sigma_1 | E \subset Z\}$.)

\item[(V)] Refining the $\sigma$-algebra: $L^q(X_2, \Sigma_T, m_2) \rightarrow
L^q(X_2, \Sigma_2, m_2): f \mapsto f$
\end{enumerate}
Since all the processes bar the one in step (II) are isometric inclusions, 
the composition operator shares the properties of the multiplier defined in that step. 

However when one passes from the setting of $L^p$ spaces to Orlicz spaces the theories 
seem to diverge somewhat. We pause to justify this claim. The following theorem describes those ``measurable transformations'' which induce composition operators from $L^{\varphi_1}$ to $L^{\varphi_2}$.

\begin{theorem}[\cite{LM}]
\label{5.1}
Let $\psi, \varphi_1, \varphi_2$ be Orlicz functions for which $\psi\circ\varphi_2 = \varphi_1$, and let $J : \M_1 \to \M_2$ be a normal Jordan $*$-morphism for which $\tau_2 \circ J$ is semifinite on $\M_1$, and $\epsilon - \delta$ absolutely continuous with respect to $\tau_1$.

Consider the following claims:
\begin{enumerate}
\item $f_J = \frac{d \tau_2\circ J}{d\tau_1} \in L^{\psi^*}(\widetilde{\M_1})$;
\item the canonical extension of $J$ to a Jordan $*$-morphism from $\widetilde{\M_1}$ to $\widetilde{\M_2}$, restricts to a bounded map $C_J$ from $L^{\varphi_1}(\widetilde{\M_1})$ to $L^{\varphi_2}(\widetilde{\M_2})$.
\end{enumerate}
The implication $(1) \Rightarrow (2)$ holds in general. If $\varphi_2$ satisfies $\Delta_2$ for all $t$, the two statements are equivalent. If $(1)$ does hold, then the norm of $C_J$ restricted to the self-adjoint portion of $L^{\varphi_1}(\widetilde{\M_1})$, is majorised by $\mathrm{max}\{1, \|f_J\|^0_{\psi^*}\}$.
\end{theorem}

For the sake of simplicity let's assume $J$ to either be an injective homomorphism, or an injective antimorphism. If one assumes that (1) above holds, the action of the ``composition operator'' can then be broken up as follows (here we have suppressed pathological technicalities for the sake of clarity):

\begin{enumerate}
\item[(O-I)] Rescaling the elements of $L^{\varphi_1}(\M_1, \tau_1)$: \newline $L^{\varphi_1}(\M_1, \tau_1) \rightarrow L^{\psi}(\M_1, \tau_1): f \to \varphi_2(f)$.

\item[(O-II)] Multiplication by the Radon-Nikodym derivative: \newline $L^{\psi}(\M_1, \tau_1) \rightarrow L^1(\M_1, \tau_1): \varphi_2(f) \to \varphi_2(f)f_J$.

\item[(O-III)] Isometric equivalence of spaces: 
$$L^1(\M_1, \tau_1) \rightarrow L^1(\M_1, \tau_2\circ J): \varphi_2(f)f_J \to \varphi_2(f).$$
$$L^1(\M_1, \tau_2\circ J) \rightarrow L^1(\mathcal{B}, \tau_2): \varphi_2(f) \to J(\varphi_2(f)) = \varphi_2(J(f)).$$
(Here $\mathcal{B}$ is the von Neumann algebra generated by $J(\M_1)$.)

\item[(O-IV)] Refining the projection lattice: \newline $L^1(\mathcal{B}, \tau_2) \rightarrow
L^1(J(\I)\M_2J(\I), \tau_2): \varphi_2(J(f)) \mapsto \varphi_2(J(f))$.

\item[(O-V)] Canonical identification: \newline $\varphi_2(f) \in L^1(J(\I)\M_2J(\I), \tau_2) \Leftrightarrow J(f) \in L^{\varphi_2}(J(\I)\M_2J(\I), \tau_2).$

\item[(O-VI)] Canonical embedding: \newline $L^{\varphi_2}(J(\I)\M_2J(\I), \tau_2) \to L^{\varphi_2}(\M_2, \tau_2).$
\end{enumerate}

Observe that steps (O-I) and (O-V) are not even linear. For us to be able to achieve the same sort of simplification of this structure that pertains in the case of $L^p$ spaces, we would need to be able to replace processes (O-I) to (O-III) above with the simple prescription that $f$ maps to $f\varphi_2^{-1}(f_T)$ and then be sure that the mapping $f\varphi_2^{-1}(f_T)\to f$ does indeed isomorphically identify $L^{\varphi_2}(\M_1, \tau_1)$ with $L^{\varphi_2}(\M_1, \tau_2\circ J)$.  That means we need access to the full strength of Proposition \ref{equivspace} for $\varphi_2$, which would of course force $\varphi_2 \in \Delta'\cap \nabla'$. In other words we would need $\varphi_2 \sim t^p$ for some $p \geq 1$. Hence it is only really when the target space is an ismorphic copy of some $L^p$ space, that we may reasonably expect to obtain such an intimate link between the theory of multipliers and of composition operators.

\section{Compactness criteria}

We take a brief look at compactness criteria for multipliers on Orlicz spaces, indicating how difficult it is for such an operator to actually be 
compact. 

\begin{lemma} \label{nonat}
Let $\M$ be a semifinite von Neumann algebra with no minimal projections.
Then any maximal abelian von Neumann subalgebra $\M_0$ of $\M$
also has no minimal projections (hence the restriction of the trace to $\M_0$ 
will still be semifinite). In addition given any projection 
$e \in \M$ with $\tau(e)=1$ and any maximal abelian subalgebra $\mathcal{N}_0$ 
of $e\M e$, the subalgebra $\mathcal{N}_0$ will still be nonatomic with respect 
to the restriction of the trace $\tau$, and will correspond to a
classical $L^\infty(\Omega, \Sigma, \rho_\tau)$, where $(\Omega,
\Sigma, \rho_\tau)$ is a nonatomic probability space and the
measure $\mu_\tau$ is defined by $\rho_\tau(E) =
\tau(\chi_E)$ for each $E \in \Sigma$. Given any Orlicz function 
$\psi$, then under the above identification, the space $L^\psi(\Omega, \Sigma, \rho_\tau)$ corresponds canonically 
to the subspace $L^\psi(\M_0,\tau|_{\mathcal{N}_0})$ of $L^\psi(\M, \tau)$. (Here 
we have departed from our usual notational convention to clarify the traces involved.) 
\end{lemma}

\begin{proof} The first claim was verified in Lemma 2.1 of \cite{GJL}. 

Next let $e$ be a projection with $\tau(e)=1$. The inclusions 
 $L^\psi(\M_0,\tau|_{\mathcal{N}_0})\subset L^\psi(e\M e, \tau|_{e\M e}) 
\subset L^\psi(\M, \tau)$, are fairly easy to verify and hence it is clear that 
we may assume without loss of generality that $e\M e = \M$. 
The commutative von Neumann subalgebra $\mathcal{N}_0$ will of course
correspond to some $L^\infty(\Omega, \Sigma, \nu)$. We may now exploit 
this corrspondence and use the restriction of $\tau$ to $\M_0 = 
L^\infty(\Omega, \Sigma, \nu)$ to define a probability 
measure $\mu$ on $(\Omega,\Sigma)$  by means of
the prescription $$\rho_\tau(E) = \varphi(\chi_E)\qquad E \in \Sigma.$$In fact 
the measure $\rho_\tau$ can be shown to have the same sets of measure zero as 
$\nu$. We may therefore replace $\nu$ by $\rho_\tau$ if necessary. Moreover the
subalgebra $\mathcal{N}_0 = L^\infty(\Omega, \Sigma, \rho_\tau)$ has no minimal
projections precisely when $(\Omega, \Sigma, \rho_\tau)$ is nonatomic.
Finally by approximating with simple functions we may show that 
$\tau(|a|) = \int |a|\rm{d}\rho_\tau$ for each $a \in \widetilde{\mathcal{N}_0}$. Given 
$\alpha > 0$ and $a \in \widetilde{\mathcal{N}_0}$ for which 
$\psi(|a|/\alpha) \in \widetilde{\mathcal{N}_0}$, we therefore have that 
$$\tau(\psi(|a|/\alpha)) = \int\psi(|a|/\alpha)\rm{d}\rho.$$It therefore follows 
that $L^\psi(\Omega, \Sigma, \rho_\tau) \equiv L^\psi(\M_0,\tau|_{\mathcal{N}_0})$ with 
preservation of norms.
\end{proof}

A simple example of a compact multiplication operator on noncommutative Orlicz spaces, is provided by the suitable direct sum of such multipliers on finite dimensional von Neumann algebras. Our primary result in this section 
essentially says that this is the only example. As such this result extends and complements similar results 
in \cite{GJL}.  

\begin{theorem} \label{compact}
Let $\M$ be a semifinite von Neumann algebra with $fns$ trace $\tau$. 
Given two Orlicz functions $\psi_1$ and 
$\psi_2$, let $g\in \tM$ be given such that the map $M_g : 
L^{\psi_1}(\tM) \rightarrow L^{\psi_2}(\tM) : a \mapsto ga$ is 
compact. Then there exists a central projection $c$ such that $gc=g$ with 
$c\M$ being a direct sum of countably many finite type I factors.
\end{theorem}

\begin{proof}
\textbf{Case 1 $(\M$ non-atomic):} Firstly suppose that $\M$ has no minimal projections. Let $e_0$ be an arbitrary projection in $\M$ with $\tau(e_0)=1$. 
We will show that the hypothesis ensures that $ge_0=0$. In view of the fact that 
$\M$ has no minimal projections, we know that 
$\I = \vee\{e\in\mathbb{P}(\M):\tau(e)=1\}$. Hence this observation 
suffices to prove that $g = 0$.

Let $\M_0$ be any maximal abelian von Neumann subalgebra of $e_0\M e_0$. 
By Lemma \ref{nonat} $$\M_0 = L^\infty(\Omega, \Sigma, \rho_\tau) 
\quad\mbox{and}\quad L^1(\M_0, \tau|_{\M_0}) = L^1(\Omega, \Sigma, \rho_\tau)$$for
some non-atomic probability space $(\Omega, \Sigma, \rho_\tau)$. Hence 
by the classical theory we can find a sequence $\{r_n\}$ of Rademacher 
type functions in $L^\infty(\Omega, \Sigma, \rho_\tau)$ (i.e. a sequence 
of self-adjoint unitaries in $\M_0$ with $\int_\Omega r_mr_nd\rho_\tau = 
\tau(r_mr_n) = \delta_{m,n}$ for each $m,n \in \mathbb{N}$).

We proceed to show that the sequence $\{r_n\}$ is weak* null in 
$L^\infty(\Omega, \Sigma, \rho_\tau)$. It suffices to show that $\langle b,
r_n\rangle = \tau(b^*r_n) \rightarrow 0$ for each $b \in 
L^1(\Omega, \Sigma, \rho_\tau)$. To this end let $b \in 
L^1(\Omega, \Sigma, \rho_\tau)$ and $\epsilon > 0$ be given. Since 
$L^2(\Omega, \Sigma, \rho_\tau)$ embeds densely into 
$L^1(\Omega, \Sigma, \rho_\tau)$, we can find $a_0 \in 
L^2(\Omega, \Sigma, \rho_\tau)$ with $\|a_0
- b\|_1 < \epsilon$. It is now trivial to conclude from this and
the H\"{o}lder inequality that $|\tau(b^*r_n) -
\tau((a_0)^*r_n)| < \epsilon$ for each $n$. Now since $\{r_n\}$ is a 
biorthogonal sequence in $L^2(\Omega, \Sigma, \rho_\tau)$, it is weakly null 
in $L^2$, which in turn ensures that $\lim_{n\to\infty}\tau((a_0)^*r_n)=0$. 
It therefore follows that $\limsup |\tau(b^*r_n)| \leq \epsilon$. Since 
$\epsilon > 0$ was arbitrary, we therefore have $\tau(b^*r_n) \rightarrow 0$ 
as required.

From classical results \cite[Cor 2.6.7]{BS} we know that $\mathcal{M}_0 = 
L^\infty(\Omega, \Sigma, \rho_\tau)$ continuously embeds into 
$L^{\psi_1}(\Omega, \Sigma, \rho_\tau))$. If therefore we compose the map $M_g$ 
firstly with the inclusion $L^{\psi_1}(\Omega, \Sigma, \rho_\tau)) = 
L^{\psi_1}(\M_0, \tau|_{\M_0}) \subset L^{\psi_1}(\tM)$ and then with the 
embedding $\M_0 \to L^{\psi_1}(\Omega, \Sigma, \rho_\tau))$, we obtain a 
compact map $$W_g:\M_0 \to L^{\psi_2}(\tM):a \to ga.$$ This compact map must 
map the weak* null sequence $\{r_n\}$ onto a norm-null sequence $\{gr_n\}$. But 
in view of the fact that $|r_n| =e_0$, we have that $$\mu_t(gr_n) = \mu_t(r_ng^*) 
= \mu_t(|r_ng^*|) \mu_t(|e_0g^*|) = \mu_t(ge_0)\quad\mbox{for all}\quad t \in [0,\infty)$$ 
and hence that $\|gr_n\|_{\psi_2} = \|\mu(gr_n)\|_{\psi_2}  = \|\mu(ge_0)\|_{\psi_2} 
= \|ge_0\|_{\psi_2} e$ for each $n$. Therefore $0=\lim_{n\to\infty} \|gr_n\|_{\psi_2} 
= \|ge_0\|_{\psi_2}$ as required.

\textbf{Case 2 ($\M$ a type $I_\infty$ factor):} Suppose now that $\M$ is a type $I_\infty$ factor. Then $\M$ may of course be represented as som $B(\mathfrak{k})$. We will show that then $M_g=0$. Suppose the contrary. Now if $M_g$ is a non-zero compact operator, then so is $M_{|g|}$. Hence we may assume $g$ to be positive. Since $g \neq 0$, we may find some $\lambda >0$ for which the spectral projection $\chi_{[\lambda, \infty)}(g) \neq 0$. Thus for some minimal subprojection $e$ of this spectral projection, we have that $ge \geq \lambda e_1$. Since $e$ is minimal and $\M$ semifinite, we must have $\tau(e)<\infty$. Now select a sequence $\{e_n\}$ of biorthogonal minimal projections in $B(\mathfrak{k})$ with $e=e_1$. Then select a sequence $\{v_n\}$ of partial isometries with $v_nv^*_n= e_1$ and $v^*_nv_n=e_n$. Then of course $\tau(v_m^*v_n) = \tau(v_nv_m^*) = \delta{nm}\tau(e_1)$. (This of course also shows that $\tau(e_n) = \tau(e_1)$.) Notice that since $\tau(e_n)<\infty$ for each $n$, each $e_n = |v_n|$ must belong to $L^\psi(\M)$, and hence so must each $v_n$. As elements of $L^2(\M,\tau)$, $\{v_n\}$ is a biorthogonal bounded sequence, and hence weakly null in $L^2(\M,\tau)$. 

By approximating with elements of $L^2(\M,\tau)\cap L^1(\M,\tau)$, the argument employed in the Case 1 to analyse the convergence properties of $\{r_n\}$, may now be extended to show that here $\{v_n\}$ is similarly weak*-null in 
$L^\infty(\M,\tau)$. But then we also have that $\tau(bv_n) \to 0$ for any $b\in L^\infty(\M,\tau)$. To see this note that for any $b \in L^\infty(\M,\tau)$, we must have that $be_1\in L^1(\M,\tau)$. Since $\{v_n\}$ is a weak*-null 
sequence in $L^\infty(\M,\tau)$, it follows that $\tau(bv_n) = \tau(be_1v_n) \to 0$ as $n\to \infty$. We have therefore managed to show that $\{v_n\}$ is a $\sigma(L^1\cap L^\infty, L^1+L^\infty)$-null sequence in 
$(L^1\cap L^\infty)(\tM)$. 

Since classically $(L^1\cap L^\infty)$ continuously embeds in $L^{\psi_1}$ \cite[\S 2.6]{BS}, the space $(L^1\cap L^\infty)(\tM)$ must similarly continuously embed into $L^{\psi_1}(\tM)$. As before we may compose thie embedding with 
the given compact map, to obtain a compact map $$W_g : (L^1\cap L^\infty)(\tM)\to L^{\psi_2}(\tM) : a \to ga.$$This compact map must map the $\sigma(L^1\cap L^\infty, L^1+L^\infty)$-null sequence $\{v_n\}$, onto a norm-null sequence. 
Hence we must have that $\|gv_n\|_\psi \to 0$. However by construction $$\mu_t(gv_n) = \mu_t(|(gv_n)^*|^2)^{1/2} = \mu_t(ge_1g^*)^{1/2} = \mu_t(ge_1) \geq \lambda\mu_t(e_1).$$Hence $\|gv_n\|_\psi$ cannot converge to 0 since we have 
that $\|gv_n\|_\psi = \|\mu(gv_n)\|\geq\lambda\|\mu(e_1)\|_\psi= \lambda\|e_1\|_\psi.$This contradiction clearly show that our starting assumption that $g \neq 0$, must be false.

\textbf{Case 3 (the general case):} Given a general semifinite algebra $\M$ and any central projection $c_0$ of $\M$, it is a simple matter to see that $c_0L^\psi(\tM) = L^\psi(c_0\tM)$ and that the action of $g$ on $L^\psi(c_0\tM)$ is 
induced by $gc_0$. So for any central projection $c_0$ for which $c_0\M$ is either a type II algebra or a type $I_\infty$ factor, we must have that $gc_0=0$. Let $c$ be the central carrier of the right support of $g$. We must 
then have that in terms of its central decomposition, $c\M$ is a direct sum of finite type $I$ factors. It remains to show that $c\M$ must be a countable direct sum of such factors. Suppose  that this is not the case. We show that this 
leads to a contradicition. For simplicity of notation we henceforth assume that $c=\I$ and simply write 
$\M = \oplus_\lambda \M_\lambda$, where each $M_\lambda$ is a finite type $I$ factor. Let $\{c_\lambda\}$ be the 
mutually orthogonal central projections for which $c_\lambda\M \equiv \M_\lambda$. 

Since $\M$ is semifinite and each $\M_\lambda$ finite dimensional, we must have that $0< \tau(c_\lambda)<\infty$ for each $\lambda$. In particular each $c_\lambda$ will then belong to $L^\psi(\tM)$. In view of the fact that the set $\{c_\lambda\}$ is uncountable, there must therefore exist some $n \in \mathbb{N}$ for which the set $S_n=\{c_\lambda : n\leq\tau(c_\lambda)<n+1\}$ is uncountable. Because $c=\I$ is the central carrier of the right support of $g$, we must have that $gc_\lambda \neq 0$ for each $\lambda$. Using the uncountability of $S_{n}$, we may furthermore find some $m\in \mathbb{N}$ for which the set $T_{nm} = \{c_\lambda\in S_n : \|gc_\lambda\|_\psi \geq \frac{1}{m}\}$ is uncountable. Corresponding to the pair $(n,m)$ we may therefore select a  sequence $c_k = c_{\lambda_k}$ of distinct $c_\lambda$'s for which $n \leq \tau(c_k) < n+1$ and $\|gc_k\|_\psi\geq \frac{1}{m}$. 

Given any $\alpha > 0$, we have by the Borel functional calculus that $\psi_1(\alpha c_k) = \psi_1(\alpha)c_k$. Combining this with the fact that $n \leq \tau(c_k) < n+1$, now yields the conclusion that $$\inf\{\alpha > 0 : \psi_1\left(\frac{1}{\alpha}\right)\leq \frac{1}{n}\} \geq \|c_k\|_{\psi_1} \geq \inf\{\alpha > 0 : \psi_1\left(\frac{1}{\alpha}\right)\leq \frac{1}{n+1}\}.$$Thus $\{c_k\}$ is a bounded sequence in $L^{\psi_1}(\tM)$. By passing to a subsequence if necessary, we may assume that the compact operator $M_g$ maps this onto a sequence $\{gc_k\}$ converging to say $a_0\in L^{\psi_2}(\tM)$ in norm and hence also in the topology of convergence in measure on $\tM$. 

Since $\|gc_k\|_\psi\geq \frac{1}{m}$ for each $k$, we must have $a_0\neq 0$. Thus we may select 
$b\in L^{\psi_1}(\tM)$ and $d\in L^{\psi_2^*}(\tM)$ so that $da_0b\neq 0$. (We may for example select $b$ to be a 
minimal subprojection of the right support of $a_0$, and $d$ a minimal subprojection of the left support of $a_0b$.) Now notice that $b$ induces a bounded multiplication operator from $L^\infty(\M,\tau)$ to $L^{\psi_1}(\tM)$, and $d$ a 
bounded multiplication operator from $L^{\psi_2}(\tM)$. The claim regarding $d$ follows from the fact that (under the Orlicz norm) $L^{\psi_2^*}(\tM)$ is the K\"othe dual of $L^{\psi_2}(\tM)$. To see the claim about $b$ notice that for 
any $f\in L^\infty(\M,\tau)$ we have that $\|bf\|_{\psi_1} = \|\mu(bf)\|_{\psi_1} \leq \|\mu(b)\|f\|_\infty\|_{\psi_1} = \|b\|_{\psi_1}\cdot\|f\|_\infty$. It follows that the composition of these multiplication operators 
$M_{dgb} = M_dM_gM_b$ yields a compact multiplication operator from $L^\infty(\M,\tau)$ to $L^1(\M,\tau)$. Next notice that by construction the sequence $\{c_k\}$ lies in $(L^1\cap L^\infty)(\tM)$. Since $\{c_k\}$ is a bounded 
biorthogonal sequence in $L^2(\M,\tau)$, essentially the same argument as in the previous two cases shows that $\{c_k\}$ is then a weak*-null sequence in $L^\infty(\M,\tau)$. The compact multiplication operator $M_{dgb}$ must 
then map this sequence onto a norm-null sequence. In other words we must have that $\{dgbc_k\}$ converges to 0 in 
$L^1$ norm and hence also in the topology of convergence in measure. However since as we saw earlier $gc_k \to a_0$ in measure, we should have that $dgbc_k = d(gc_k)b \to da_0b$ in measure. But this can't be since by construction 
$da_0b \neq 0$. This contradiction then serves to establish that our initial assumption that $c\M$ is a direct sum of uncountably many finite type I factors, must be false.
\end{proof}

\begin{remark}
In view of the theme of this paper, the above result was formulated for Orlicz spaces. It is however worth pointing out that with minor modifications the techniques carry over to any pair of Banach Function spaces which appear as intermediate spaces of the Banach couple $(L^\infty(\M,\tau), L^1(\M,\tau))$. 
\end{remark}

\end{document}